\documentclass[11pt, a4paper, english, reqno]{amsart}
\usepackage{amsmath} 
\usepackage{amsthm} 
\usepackage{amssymb} 
\usepackage[dvipsnames]{xcolor}
\usepackage{amscd} 
\usepackage{mathtools}
\usepackage[normalem]{ulem}
\usepackage{mlmodern}
\usepackage{subcaption}

\usepackage{array} 

\usepackage[cal=boondoxo,scr=euler]{mathalfa}
\usepackage[linktocpage]{hyperref} 
\usepackage{cleveref} 
\usepackage{caption} %
\usepackage{graphics,graphicx} 
\usepackage{tikz,tikz-cd} 

\usetikzlibrary{fit,matrix,graphs,graphs.standard,calc,shapes.geometric, arrows.meta}
\usetikzlibrary{decorations.pathreplacing,}
\usetikzlibrary{automata}

\usepackage[shortlabels]{enumitem} 
\usepackage{xcolor}

\usepackage[colorinlistoftodos,bordercolor=orange,backgroundcolor=orange!20,linecolor=orange,textsize=scriptsize]{todonotes}

\DeclareMathAlphabet{\mathsf}{OT1}{\sfdefault}{m}{n}

\newcommand{\nocontentsline}[3]{}
\newcommand{\tocless}[2]{\bgroup\let\addcontentsline=\nocontentsline#1{#2}\egroup}

\usepackage[margin=1.28in]{geometry}
\usepackage{verbatim}

\usepackage{scalerel}

\makeatletter
\def\dual#1{\expandafter\dual@aux#1\@nil}
\def\dual@aux#1/#2\@nil{\begin{tabular}{@{}c@{}}#1\\#2\end{tabular}}
\makeatother

\makeatletter
\@namedef{subjclassname@2020}{\textup{2020} Mathematics Subject Classification}
\makeatother

\DeclareMathAlphabet{\amathbb}{U}{bbold}{m}{n}

\hypersetup{
    colorlinks = true,
    linkbordercolor = {white},
    linkcolor = {BrickRed},
    anchorcolor = {black},
    citecolor = {BrickRed},
    filecolor = {cyan},
    menucolor = {BrickRed},
    runcolor = {cyan},
    urlcolor = {black}
}

\usetikzlibrary{automata}

\newtheoremstyle{teoremas}
{9pt}
{9pt}
{\itshape}
{}
{\bfseries}
{}
{.5em}
{}

\theoremstyle{teoremas}
\newtheorem{theorem}{Theorem}[section]
\newtheorem{corollary}[theorem]{Corollary}
\newtheorem{lemma}[theorem]{Lemma}

\newtheoremstyle{definition}
{8pt}
{8pt}
{}
{}
{\bfseries}
{}
{.5em}
{}

\theoremstyle{definition}
\newtheorem{definition}[theorem]{Definition}
\newtheorem{conjecture}[theorem]{Conjecture}

\newtheorem{problem}[theorem]{Problem}

\newtheorem{remark}[theorem]{Remark}

\crefname{theorem}{theorem}{theorems}
\Crefname{theorem}{Theorem}{Theorems}
\crefname{lemma}{lemma}{lemmas}
\Crefname{lemma}{Lemma}{Lemmas}
\crefname{proposition}{proposition}{propositions}
\Crefname{proposition}{Proposition}{Propositions}

\tikzstyle{rectan} = [rectangle, rounded corners, 
minimum width=1.5cm, 
minimum height=0.75cm,
text width=3cm,
text centered, 
draw=black,
font = \normalfont,
]

\tikzstyle{ghost} = [circle, 
minimum width=1pt, 
minimum height=1pt,
text width=1pt,
text centered, 
draw=black,
font = \footnotesize,
]

\newcommand{\HA}{A}
\newcommand{\HB}{B}

\newcommand{\Hilb}{\operatorname{Hilb}}

\renewcommand{\H}{\mathrm{H}}
\newcommand{\CH}{\mathrm{CH}}
\newcommand{\G}{\mathcal{G}}

\DeclareMathOperator{\Bell}{\mathrm{Bell}}
\newcommand{\Mbar}{\overline{\mathcal{M}}}

\renewcommand{\hat}{\widehat}
\renewcommand{\emptyset}{\varnothing}

\AtBeginDocument{%
   \def\MR#1{}
}

\title[The Poincar\'e polynomial of $\Mbar_{0,n}^B$]{The Poincar\'e polynomial\\ of the type B analogue of $\Mbar_{0,n+1}$}

\author[Ferroni, Pagaria, and Vecchi]{Luis Ferroni, Roberto Pagaria, and Lorenzo Vecchi}

\address{(L.~Ferroni)
  Universit\`a di Pisa, Pisa, Italy
}
\email{luis.ferroni@unipi.it}

\address{(R.~Pagaria)
    Universit\`a di Bologna, Bologna, Italy}
\email{roberto.pagaria@unibo.it}

\address{(L. Vecchi)
  KTH Royal Institute of Technology, Stockholm, Sweden
}

\email{lvecchi@kth.se}

\begin{document}

\allowdisplaybreaks

\begin{abstract}
    We establish formulas for the Poincaré polynomial of the type B analogue of the Deligne--Knudsen--Mumford moduli space of rational curves with $n$ marked points, providing type B counterparts to results by Keel, Manin, Getzler, and Yuzvinsky. We establish functional and differential equations satisfied by the bivariate exponential generating function of these polynomials. We show how this generating function relates to the classical one in type A. We deduce the $\gamma$-positivity of these polynomials via a quadratic recursion and discuss a type B analogue of a formula found by Aluffi, Marcolli and Nascimento in type A.
\end{abstract}

\subjclass[2020]{Primary: 05B35, 13D40, 14C15. Secondary: 16S37}

\keywords{Matroids, hyperplane arrangements, Chow rings, building sets, Hilbert series, log-concavity, moduli space of stable pointed curves}

\maketitle

\section{Introduction}

For each $n\geq 3$, let $\overline{\mathcal{M}}_{0,n}$ be the Deligne--Knudsen--Mumford moduli space of stable rational curves, with $n$ marked points. This moduli space is a ubiquitous object, featured in many areas of mathematics ranging from algebraic geometry to combinatorics. The cohomology of $\overline{\mathcal{M}}_{0,n}$ and, in particular, its Betti numbers have been a recurring subject of study over the past three decades \cite{keel,manin,Getzler94,yuzvinsky}. 

Recently, Aluffi, Marcolli, and Nascimento \cite{aluffi-marcolli-nascimento} found a new and explicit formula for the Poincar\'e polynomial associated to $\Mbar_{0,n}$. Their formula is written in terms of Stirling numbers of the first and second kind.
Their formula is in fact a consequence of one of Getzler's identities, as was shown by the authors of the present paper together with Eur and Matherne \cite{eur-ferroni-matherne-pagaria-vecchi}. Indeed, it can be deduced using Lagrange inversion, Poincaré duality, and the identity \cite[Equation~(8)]{eur-ferroni-matherne-pagaria-vecchi}.

The main goal of this article is to produce analogues of the aforementioned formulas for a ``type B'' analogue of $\Mbar_{0,n}$, which we denote by $\Mbar_{0,n}^B$. 
It can be constructed as the wonderful compactification \cite{deconcini-procesi} of the type B braid arrangement with respect to the minimal building set.
Similar generalizations have already appeared in the literature. In \cite{blankers-clader-halacheva-liu-ross}, the authors study a moduli space denoted by $\Mbar_n^r$ as the wonderful compactification of $r$-braid arrangements, of which our type B arrangement is a specific instance when $r=2$. The structure of the rational cohomology of $\Mbar_{0,n}^B(\mathbb{R})$, and more generally for other Coxeter types, has been studied in \cite{henderson-rains}.


From now on, we study the Betti numbers of the cohomology of the moduli space $\Mbar^B_{0,n}$. Since the cohomology is concentrated in even degrees, we write the Poincaré polynomial as
\[
P_{\Mbar^B_{0,n}}(x) = \sum_{i\geq 0}\dim \H^{2i}(\Mbar_{0,n}^B, \mathbb Q)\ x^{i}.
\]

There have been only a few attempts in the literature to carry out an explicit computation of these Betti numbers. Notably, a result by Yuzvinsky appearing in \cite[Theorem~5.1]{yuzvinsky} yields, via a combinatorial interpretation, the first few values of $P_{\Mbar^B_{0,n}}(x)$. These are reproduced from \cite[p.~329]{yuzvinsky}. 

\[ P_{\Mbar^B_{0,n}}(x) = \begin{cases}
    1 & n = 1,\\
    x+1 & n = 2,\\
    x^2+8x+1 & n = 3,\\
    x^3+35x^2+35x+1 & n = 4,\\
    x^4+122x^3+482x^2+122x+1 & n = 5.
\end{cases}\]

To describe $P_{\Mbar^B_{0,n}}(x)$ compactly, we encode these polynomials via their respective exponential generating functions,

\[
A(t,x) = 1 + t + \sum_{n\geq 2}P_{\Mbar_{0,n+1}}(x)\frac{t^n}{n!}
\]
and
\[
B(t,x) = 1 + \sum_{n\geq 1}P_{\Mbar^B_{0,n}}(x)\frac{t^n}{2^n\ n!}
\]
and show how these two series are related to each other.

\begin{theorem}\label{thm:main1_func_eq_A_B}
The generating functions in type $A$ and $B$ are related by
\begin{align*}
B &= \frac{1-x}{\HA^{\frac{x-1}{2}}-x},\\
&= \frac{x \HA + \sqrt{x^2A^2 + \HA (1-x)(1+x+xt)}}{1+x+xt}.
\end{align*}
Solving for $A$ the last equation gives the equivalent form:
\[ A = \frac{(1+x+xt)B^2}{1-x+2xB}.\]
\end{theorem}

These identities are proved in the main body as Theorems~\ref{thm:HB_from_HA} and \ref{thm:HB_from_HA2}, the first one appeared (after a transformation) in \cite{gaiffi2} and \cite{henderson}. After we are able to relate the generating functions $B$ and $A$, we can employ some well-known results in the literature to provide identities only involving $B$ without making reference to $A$. For example, we have the following functional equation, which will be found later as Theorem~\ref{thm:functional-B}.

\begin{theorem}\label{thm:main3_func_eq_B}
    The generating function $\HB$ satisfies the functional equation
    \begin{equation*}
    \left( \frac{B^2(1+x+xt)}{1-x+2xB} \right)^{\frac{x-1}{2}} = \frac{1-x}{\HB} +x. 
    \end{equation*}
\end{theorem}

Furthermore, by leveraging well-known differential equations that $A$ satisfies, one can express $B$ as a solution to a differential equation as well. The next result is established as Theorem~\ref{thm:diff-eq-gf-B}.

\begin{theorem}\label{thm:main2_diff_eq_A_B}
The exponential generating function $B$ satisfies the differential equation
\begin{equation*}
\frac{\partial \HB}{\partial t} = \frac{1}{2} \frac{B-xB+xB^2}{1+x+xt-xA}
\end{equation*}
with initial condition $B(0,x)=1$.
\end{theorem}

Finally, Theorem \ref{thm:main2_diff_eq_A_B} lets us provide a nice quadratic recursion for $B_n(x) \coloneq P_{\Mbar^B_{0,n}}(x)$ in terms of $A_n(x) \coloneq P_{\Mbar_{0,n+1}}(x)$ (see later Theorem~\ref{thm:quadratic-recursion}).

\begin{theorem}\label{thm:main4_quadratic_recursion}
The following recursive relation holds:
\begin{equation*}
B_{n+1}=B_n + x \sum_{j=1}^{n} \binom{n}{j} B_{j} B_{n-j} + x \sum_{j=2}^{n} \binom{n}{j}2^j A_{j} B_{n+1-j},
\end{equation*}
with initial condition $B_0=1$.
\end{theorem}

We can use Theorem \ref{thm:main4_quadratic_recursion} to readily prove the following corollary, which in the type A case was noted explicitly in recent work by Aluffi, Chen and Marcolli \cite[Theorem~1.2]{aluffi-marcolli-chen}.

\begin{corollary}\label{cor:main5_gamma_positive}
The polynomials $B_n(x)$ are $\gamma$-positive for every $n\geq 0$.
\end{corollary}

Note that the Hard Lefschetz theorem on the cohomology $\H^{\bullet}(\Mbar_{0,n}^B)$ guarantees the unimodality of the polynomials $B_n$. Moreover, $\gamma$-positivity implies unimodality (and is much stronger): the quadratic recursion allows us to access this stronger property of the sequence of Betti numbers of $\Mbar_{0,n}^B$.

A non-recursive formula for $B_n(x)$, which can be regarded in a precise way as an analogue of a type A formula by Aluffi, Marcolli, and Nascimento \cite{aluffi-marcolli-nascimento}, can be obtained from Theorem \ref{thm:main1_func_eq_A_B} as follows using \emph{Bell polynomials}.

\begin{theorem}
Let $\overline{\chi}^B_n$ denote the reduced characteristic polynomial of the $n$-dimensional type B braid arrangement. Then
\[
    \HB_n(x) = \sum_{k=1}^n 2^{n-k}\left( \sum_{\ell=0}^k \ell! \Bell_{k,\ell} (\overline{\chi}_{1}^B, \overline{\chi}_{2}^B, \dots ) \right)  \Bell_{n,k} (A_1 ,A_2, \dots).
\]
\end{theorem}

The above statement corresponds to Theorem~\ref{thm:iterative-B}.

\section{The lattice of signed partitions}
In this section, we review the poset-theoretic model that we will use to begin studying the cohomology of $\Mbar_{0,n}^B$. 
The rank $n$ braid arrangement of type A in $\mathbb{C}^{n+1}$, denoted by $\mathcal A_{n}$, is the arrangement consisting of all the hyperplanes
\[
H_{ij} \coloneq \left\{x \in \mathbb C^{n+1} : x_i = x_j \right\} \qquad \text{for each $1\leq i < j \leq n+1$}.
\]
The moduli space $\Mbar_{0,n+1}$ can be described as the wonderful compactification of $\mathcal A_n$ with respect to the minimal building set. Since we use the geometric viewpoint solely as motivation, we refer to \cite{deconcini-procesi} for more details about wonderful compactifications.
An important combinatorial invariant of a hyperplane arrangement $\mathcal{A}$ is its \emph{lattice of intersections} $\mathcal{L}(\mathcal{A})$, i.e., the poset of all subspaces of the ambient space that can be written as the intersection of hyperplanes in $\mathcal A$ ordered by reverse inclusion. It is well known that the lattice of intersections of $\mathcal A_n$ is isomorphic to the partition lattice $\Pi^A_n$, i.e., the poset of all partitions of $[n]$ ordered by coarsening.

We want to consider a similar combinatorial model for $\Mbar_{0,n}^B$.
The rank $n$ braid arrangement of type B in $\mathbb{C}^{n}$, denoted by $\mathcal B_n$, is the arrangement consisting of all the hyperplanes
        \begin{align*} 
            H^{\pm}_{ij} := \left\{ x\in \mathbb{C}^n : x_i \pm x_j = 0\right\} \qquad &\text{for each $1\leq i < j \leq n$},\\*
            H_{ii} := \left\{ x\in \mathbb{C}^n : x_i = 0\right\} \qquad &\text{for each $1\leq i \leq n$}.
        \end{align*}
The moduli space $\Mbar_{0,n}^B$ can be described as the wonderful compactification of this hyperplane arrangement with respect to the minimal building set. 
In the case of the type $B$ braid arrangement, the lattice of intersections is isomorphic to the lattice of signed partitions, which we now describe.

Let $\langle n \rangle$ be the set $\{1,2, \dots, n, \overline{1}, \dots, \overline{n}\}$ obtained as the disjoint union of two copies of the integers between $1$ and $n$.
This set has a natural involution $i\mapsto \overline{i}$, where we set $\overline{\overline{i}}:= i$ for each $i\in \{1,\ldots,n\}$. For each subset $B \subseteq \langle n \rangle$, we define the \emph{conjugate} of $B$ as the set $\overline{B}= \{ \overline{i} \mid i \in B\}$.

\begin{definition}
    A \emph{signed partition} $\pi$ is a partition of the set $\langle n \rangle$ such that for each block $B \in \pi$ the conjugate block $\overline{B}$ is in $\pi$ and there exists at most one block $Z \in \pi$ such that $Z=\overline{Z}$.
\end{definition}
If $\pi$ is a signed partition and $Z$ is a block of $\pi$ such that $Z=\overline{Z}$, then we call $Z$ the \emph{zero block} of $\pi$. We will also denote by $\ell(\pi)$ half the number of non-zero blocks. In this way, $\pi = \{ Z, B_1, \dots, B_{\ell(\pi)}, \overline{B}_1, \dots , \overline{B}_{\ell(\pi)} \}$.
We write $\sigma \vdash [ n ]$ for $\sigma$ a set partition of $[n]:=\{1,\ldots,n\}$ and $\pi \vdash \langle n \rangle$ for $\pi$ a signed partition of $\langle n\rangle$.

\begin{definition}
Let $n\geq 1$. The poset of signed partitions of $\langle n\rangle$ ordered by coarsening is called the lattice of signed partitions, denoted by $\Pi^B_n$.
\end{definition}

The isomorphism between the poset of intersections and $\Pi^B_n$ is the following: to a given signed partition $\pi = \left\{Z, B_1,\ldots, B_{k},\overline{B}_1,\ldots, \overline{B}_k \right\}$ we associate the subspace $X(\pi)$ described by the equations
\[\begin{cases}
x_i = 0, &  i \in Z \\
x_i = x_j, & \text{if } i, j \text{ are in the same block,}\\
x_i = -x_j, & \text{if } i, \overline{j} \text{ are in the same block.}
\end{cases}
\]
The minimal element of $\Pi^B_n$, denoted by $\hat{0}$ is then given by the signed partition with $2n$ singletons, while the maximal element $\hat{1}$ is given by $\left\{\langle n\rangle\right\}$. For a given signed partition $\pi = \left\{Z, B_1, \ldots, B_k, \overline{B}_1, \ldots, \overline{B}_k \right\}$, we observe that
\begin{equation}\label{eq:intervals_PiB}
\left[\hat{0},\pi \right] \cong \Pi^B_{\frac{|Z|}{2}} \times \Pi^A_{|B_1|} \times \cdots \times \Pi^A_{|B_k|}
\qquad \text{and} \qquad
\left[\pi,\hat{1}\right] \cong \Pi^B_{k};
\end{equation}
see \cite{dowling1973,zaslavsky}.
It is also well known that the characteristic polynomial of $\Pi^B_n$ is
\[
\chi^B_n(x) = \prod_{i=0}^{n-1} (x- (2i+1))
\]
and $\overline{\chi}^B_n(x) = \frac{\chi^B_n(x)}{x-1}$ is its reduced characteristic polynomial \cite[Theorem~3]{zaslavsky}. We extend this definition by setting $\overline{\chi}^B_0 \coloneq -1$; this agrees with the choice made in \cite{ferroni-matherne-vecchi}.
As presented in the introduction, we will henceforth use the following notation to denote the Poincar\'e polynomials of the cohomologies of $\Mbar_{0,n+1}$ and its type B counterpart $\Mbar^B_{0,n}$. 
\[A_n(x) \coloneq P_{\Mbar_{0,n+1}}(x) \qquad \text{ and } \qquad B_n(x) \coloneq P_{\Mbar_{0,n}^B}(x).\]
We set $B_0=1$.
(The reader is warned to bear in mind the difference in the indexing between these two Poincar\'e polynomials.) Whenever needed, we will allude to the cohomologies themselves by the notation $\H^{\bullet}(\Mbar_{0,n+1})$ and $\H^{\bullet}(\Mbar^B_{0,n})$, respectively.

In a prequel paper that we wrote together with Eur and Matherne \cite{eur-ferroni-matherne-pagaria-vecchi}, we proved a general formula for the polynomial \(\H_L^{\G}(x)=\Hilb(\CH(L,\G),x)\), i.e, the Poincar\'e polynomial of the cohomology of the wonderful compactification of an arrangement with lattice of intersections \(L\) using an arbitrary building set \(\mathcal G\)---moreover, the framework is much more general and works even for polymatroids. 
Using this interpretation, one can write,
    \[
    \H_{\Pi^B_j}^{\G_{\min}}(x)=B_j(x),\qquad
    \H_{\Pi^A_m}^{\G_{\min}}(x)=A_m(x),
    \]
where \(\Pi^B_m\) and \(\Pi^A_m\) are the braid matroids of type \(B\) and type \(A\), respectively, and \(\G_{\min}\) is the minimal building set.
In the context of the present paper, we may specialize  \cite[Theorem~3.6]{eur-ferroni-matherne-pagaria-vecchi} in order to obtain the following formula. 

\begin{theorem}\label{thm:B_n_as_chow_polynomial}
    For every $n\geq 1$,
    \[
    \sum_{j=0}^n \binom{n}{j} B_{j}(x) \sum_{\substack{\sigma \vdash [n-j]}}2^{n-j-\ell(\sigma)}\prod_{i=1}^{\ell(\sigma)}A_{|B_i|}(x)\ \overline{\chi}^B_{\ell(\sigma)}(x) = 0 .
    \]
\end{theorem}

\begin{proof}
    
    By \cite[Theorem~3.6]{eur-ferroni-matherne-pagaria-vecchi} and using the notation introduced there we can write
    \[
    \sum_{\pi \in \Pi^B_n} \H^{\mathcal G_{\min}|_{\pi}}_{[\hat{0},\pi]}(x) \overline{\chi}^{\mathcal G_{\min}/\pi}_{[\pi,\hat{1}]}(x) = 0.
    \]
    We now recall that a signed partition $\pi$ is determined by the choice of the zero block $Z$ of size $2j$, a partition $\sigma$ of $[n]\setminus Z$ and $2^{n-j-\ell(\sigma)}$ signs.
    In particular,
    \[
    \H^{\G_{\min}|_{\pi}}_{[\hat{0},\pi]}(x)
      =B_j(x)\prod_{i=1}^{\ell(\pi)}A_{|B_i|}(x),\qquad
    \overline{\chi}^{\G_{\min}/\pi}_{[\pi,\hat{1}]}(x)
      =\overline{\chi}^B_{\ell(\pi)}(x).
    \]
    We conclude using \eqref{eq:intervals_PiB} after observing that $\H^{\mathcal G_{\min}}_{P\times Q}(x) = \H^{\mathcal G_{\min}|_P}_{P}(x)\cdot\H^{\mathcal G_{\min}|_Q}_{Q}(x) $.
\end{proof}

\section{A toolkit on generating functions}
We now recall some classical facts on generating functions. 
Let $R$ be a ring containing $\mathbb{Q}$, and consider a sequence $(f_n)_{n\geq0}$ in $R$. We define the \emph{type $B$} exponential generating function of $(f_n)$ as the formal power series
\[ F^B(t) \coloneqq \sum_{n \geq 0} f_n \frac{t^n}{2^n n!} \in R[[t]].\]
Note that the classical exponential generating function (from now on alluded to as the type A exponential generating function), defined by
\[ F^A(t) \coloneqq \sum_{n \geq 0} f_n \frac{t^n}{n!} \in R[[t]]\]
encodes exactly the same information as $F^B(t)$ via the equation $F^B(2t) = F^A(t)$.

Recall that for fixed integers $n$ and $k$, the \emph{Bell polynomial} $\Bell_{n,k}\in \mathbb{Q}[y_1,\ldots,y_{n-k+1}]$ is defined as
\[ \Bell_{n,k}(y_1,\ldots,y_{n-k+1}) = \sum_{\substack{\sigma \vdash [n] \\ \ell(\sigma)=k}} \prod_{i=1}^{k} y_{|\sigma_i|},\]
where the above sum runs over all set partitions of $[n]$ with exactly $k$ parts. For example, for $n = 6$ and $k = 2$, we consider the partitions of the set $\{1,2,3,4,5,6\}$ into two parts: $10$ such partitions correspond to two blocks of size $3$, while $15$ partitions correspond to two blocks of sizes $2$ and $4$ respectively, and $6$ partitions correspond to a singleton plus a block of size~$5$. Thus $\Bell_{6,2}(y_1,y_2,y_3,y_4,y_5) = 10y_3^2+15y_2y_4+6y_1y_5$.

Bell polynomials can be used to provide compact expressions for the composition of two generating functions. 

\begin{theorem}[The compositional formula]
\label{thm:compositional_formula}
Let $F(t)=\sum_{n \geq 1} f_n \frac{t^n}{n!}$ and $G(t)=\sum_{n \geq 0} g_n \frac{t^n}{n!}$  be two series in $R[[t]]$ such that $F(0) = 0$. Then the composition $G(F(t)) =\colon H(t) = \sum_{n\geq 0 }h_n \frac{t^n}{n!}$ has coefficients
\[
h_n = \sum_{k=0}^ng_k \Bell_{n,k} \left(f_1, f_2, \ldots, f_{n-k+1} \right)
\]
for $n\geq 0$.
\end{theorem}

The above formulation is equivalent to the statement \cite[Theorem~5.1.4]{stanley-ec2}. For this precise version, which is known as (a special case of) the \emph{Fa\`a di Bruno formula}, we refer to \cite[Theorem~A, p.~137]{comtet}. 

\begin{theorem}[Lagrange inversion formula]
\label{thm:Lagrange_inversion}
Let $F(t)=\sum_{n \geq 1} f_n \frac{t^n}{n!}$ and $G(t)=\sum_{n \geq 1} g_n \frac{t^n}{n!}$  be two series in $R[[t]]$ such that $f_1 \neq 0$ and $G(F(t))=t$. Then 
\[g_n = \frac{1}{f_1^n}\sum_{k=0}^{n-1} \Bell_{n+k-1,k} \left( 0, -\frac{f_2}{f_1}, -\frac{f_3}{f_1}, -\frac{f_4}{f_1}, \dots, - \frac{f_n}{f_1}\right) \]
for $n \geq 1$.
\end{theorem}

The above statement can be found in \cite[{Theorem~E, p.~151, eq.~8g}]{comtet}. It is formally equivalent to the formulation in \cite[Theorem~5.4.2]{stanley-ec2} (specializing that statement to $k=1$). 

\begin{theorem}\label{thm:reciprocal_series}
Let $F(t)=\sum_{n \geq 0} f_n \frac{t^n}{n!}$ be a series in $R[[t]]$ such that $f_0 \neq 0$. Then its multiplicative inverse $1/F(t) = \sum_{n\geq 0}g_n \frac{t^n}{n!}$ exists uniquely and its coefficients are given by 
\[
g_n = \frac{1}{f_0}\sum_{k\geq 0}(-1)^k\, k ! \, \Bell_{n,k} \left(\frac{f_1}{f_0}, \frac{f_2}{f_0}, \dots , \frac{f_{n-k+1}}{f_0}  \right)
.\]
\end{theorem}

\begin{proof}
The above follows from the compositional formula by taking \[G(t) = \frac{1}{f_0+t} = \frac{1/f_0}{1 + t/f_0} = \frac{1}{f_0} \sum_{k\geq 0} (-1)^k\frac{1}{f_0^k} t^k,\] and the function $\widehat{F}(t) = F(t) - f_0$, which by definition satisfies $\widehat{F}(0) = 0$. We have:
\[ G(\widehat{F}(t)) = \frac{1}{f_0+\widehat{F}(t)} = \frac{1}{f_0 + (F(t)-f_0)} = \frac{1}{F(t)}.\]
After substituting the coefficients of the generating functions of $G$ and $\widehat{F}$ we arrive at the claimed identity.
\end{proof}

We now proceed to extend these results to the framework of type B exponential generating functions. The following is the type B analogue of the compositional formula appearing in Stanley's \cite[Theorem~5.1.4]{stanley-ec2}. 

\begin{theorem}
\label{thm:B_compositional_formula}
    Let $(f_n),(g_n),(k_n)$ be three sequences in $R$ such that $f_0=0$.
    Let $G,K$ be the type $B$ exponential generating functions of $(g_n)$ and $(k_n)$, respectively, and $F$ the type A exponential generating function of $(f_n)$.
    Define
    \[h_n \coloneqq \sum_{\pi \vdash \langle n \rangle } k_{\lvert Z\rvert/2} g_{\ell(\pi)} \prod_{i=1}^{\ell(\pi)} f_{\lvert B_i\rvert}. \]
    Then the type $B$ exponential generating function of $(h_n)$ is given by
    \[ H(t)= K(t) G(F(t)). \]
\end{theorem}
\begin{proof}
It is easy to see that a signed partition $\pi$ is determined by the choice of a subset of $[n]$ of cardinality $z \coloneqq \frac{\lvert Z \rvert}{2}$, a partition of $[n-z]$, and $2^{n-z-\ell(\pi)}$ signs.
Therefore,
\[h_n=\sum_{z=0}^{n}\binom{n}{z} k_z \sum_{\sigma \vdash [n-z]} 2^{n-z-\ell(\sigma)} g_{\ell(\sigma)} \prod_{i=1}^{\ell(\sigma)} f_{\lvert B_i\rvert}\]
and the generating function can be factored as
\[ H(t)= \left(\sum_{z\geq 0} k_z \frac{t^z}{2^z z!}\right)\left( \sum_{u\geq 0} \sum_{\sigma \vdash [u]} \frac{g_{\ell(\sigma)}}{2^{\ell(\sigma)}} \prod_{i=1}^{\ell(\sigma)} f_{\lvert B_i\rvert} \frac{t^u}{u!} \right). \]
Using the classical compositional formula (see e.g.\ \cite[Theorem 5.1.4]{stanley-ec2}) we obtain the desired result.
\end{proof}

Taking $K=1$, this can be immediately recast in terms of Bell polynomials as follows, thus giving a type B analogue of Theorem~\ref{thm:compositional_formula}. 
Alternatively, this can be derived directly from that theorem. The formula with a general $K$ is needed later to prove Theorem~\ref{thm:HB_from_HA}.

\begin{corollary}\label{coro:compositional_formula_type_B}
Let $F(t) = \sum_{n\geq 1} f_n \frac{t^n}{n!}$ and $G(t) = \sum_{n\geq 0}g_n \frac{t^n}{2^n n!}$ be two series in $R[[t]]$.
Then the composition $G(F(t)) =: H(t) = \sum_{n\geq 0} h_n \frac{t^n}{2^n n!}$ as a type B exponential generating function has coefficients
\[
h_n = \sum_{k=0}^n 2^{n-k}g_k \Bell_{n,k}(f_1,f_2,\ldots).
\]
\end{corollary}

\section{Type B analogues for the Poincar\'e polynomial of \texorpdfstring{$\Mbar_{0,n+1}$}{Mbar0n+1}}

\noindent Let $\HA(t,x)$ be the exponential generating function of $A_n(x)=P_{\Mbar_{0,n+1}}(x)$:
\begin{align*}
    \HA(t,x) &=1+ t + \sum_{n\geq 2} A_n \frac{t^n}{ n!} \\ 
    & = 1 + t + \frac{t^2}{2} + (x+1) \frac{t^3}{3!} + (x^2+5x+1)\frac{t^4}{4!} + (x^3+16x^2+16x+1) \frac{t^5}{5!} \\
    &\qquad+ (x^4+42 x^3 + 127 x^2 + 42 x +1) \frac{t^6}{6!}+ \dots
\end{align*}

There are many implicit and explicit ways of expressing the above generating function. Notably, the following result of Manin provides both a functional equation and a differential equation whose solutions are given by $\HA(t,x)$.

\begin{theorem}[{\cite[Theorem~0.3.1]{manin}}]
The generating function $A(t,x)$ can be obtained as the only solution in $1+t+t^2\mathbb{Q}[x][[t]]$ of any of the following two equations:
    \begin{equation}\label{eq:Manin A}
        \HA^x-x^2 \HA =1-x^2+(1-x)xt.
    \end{equation}
    \begin{equation}\label{eq:Manin diff A}
        \HA = (1+x+xt-x\HA)\frac{\partial}{\partial t}\HA.
    \end{equation}
\end{theorem}

Our primary goal in this section is to provide type $B$ analogues of the above two equations. To this end, let us denote by $\HB(t,x)$ the type $B$ exponential generating function of $B_n(x) =  P_{\Mbar_{0,n}^B}(x)$:
\begin{align*}
    \HB(t,x) &= 1+ \sum_{n\geq 1} B_n \frac{t^n}{2^n n!} \\
    &= 1+\frac{t}{2} + (x+1) \frac{t^2}{2^2 2!} + (x^2+8x+1)\frac{t^3}{2^3 3!} + (x^3+35x^2+35x+1) \frac{t^4}{2^4 4!} \\
    &\qquad + (x^4+122 x^3 + 482 x^2 +122 x +1) \frac{t^5}{2^5 5!}+ \dots
\end{align*}

\subsection{Formulas relating A and B}

We provide three different expressions for the generating function $B$ in terms of 
$A$. In order to simplify our subsequent computations, we prove the following preliminary result.

\begin{lemma}\label{lemma:gen_fun_char_poly_B}
    The type $B$ exponential generating function of the reduced characteristic polynomial $\overline{\chi}_n^B$ is 
    \[
        \overline{\chi}^B (t,x) \coloneq \sum_{n\geq 0}\overline{\chi}^B_n(x) \frac{t^n}{2^n n!} = \frac{(1+t)^{\frac{x-1}{2}}-x}{x-1}.
    \]
\end{lemma}

\begin{proof}
By definition the generating function is
\begin{align*}
-1+\sum_{n \geq 1} \overline{\chi}_n^B(x) \frac{t^n}{2^n n!} 
&= -1 + \frac{\sum_{n \geq 1} \chi_n^B(x) \frac{t^n}{2^n n!}}{x-1}
= -1 + \frac{\sum_{n \geq 1} \prod_{i=0}^{n-1} \left( \frac{x-1}{2}-i \right) \frac{t^n}{n!}}{x-1} \\
&= -1 + \frac{\sum_{n \geq 1} \binom{(x-1)/2}{n} t^n}{x-1} = -1+\frac{(1+t)^{\frac{x-1}{2}}-1}{x-1},
\end{align*}
which is the claimed function.
\end{proof}

\begin{remark}
    As we explained in \cite[Proposition~4.5(i)]{eur-ferroni-matherne-pagaria-vecchi}, the generating function
    \[\overline{\chi}^A(t,x):=\sum_{n \geq 1} \overline{\chi}_n^A(x) \frac{t^n}{n!} = -t + \frac{(1+t)^x-1-xt}{x(x-1)}\]
    satisfies the property that the compositional inverse of $-\overline{\chi}^A(t,x)$ is precisely the generating function $A-1$. The reader \emph{should not} be misled to think that the compositional inverse of $-\overline{\chi}^B(t,x)$ leads so directly to the generating function $B$. The precise relationship that we obtain between $\overline{\chi}^B(t,x)$ and $B$ is stated in the next theorem.
\end{remark}

The following result makes clear the power of Theorem~\ref{thm:B_n_as_chow_polynomial}. This is a formula that plays a crucial role in deriving most of our results.

\begin{theorem}
    The generating functions $\HA$ and $\HB$ are related by:
\label{thm:HB_from_HA}
\begin{equation}
B = \frac{1-x}{\HA^{\frac{x-1}{2}}-x} = \frac{-1}{\overline{\chi}^B(A-1,x)}\label{eq:HB_from_HA}. 
\end{equation}
\end{theorem}
\begin{proof}
    In Theorem~\ref{thm:B_compositional_formula}, take
    \[
    K(t)=B(t,x),\qquad G(t)=\overline{\chi}^B(t,x),\qquad
    F(t)=A(t,x)-1.
    \]
    Explicitly,
    \begin{align*}
    &B(t,x)\overline{\chi}^B(A(t,x)-1,x)\\
    &\quad=-1+\sum_{n\geq1}\frac{t^n}{2^n n!}
    \left[\sum_{j=0}^n\binom{n}{j}B_j(x)
    \sum_{\sigma\vdash[n-j]}2^{n-j-\ell(\sigma)}
    \prod_{i=1}^{\ell(\sigma)}A_{|B_i|}(x)\,
    \overline{\chi}^B_{\ell(\sigma)}(x)\right].
    \end{align*}
    We use Lemma~\ref{lemma:gen_fun_char_poly_B} and Theorem \ref{thm:B_n_as_chow_polynomial} to simplify the left and right-hand side, respectively. This yields the identity 
    \begin{equation*}
        B \cdot \frac{A^{\frac{x-1}{2}}-x}{x-1}=-1,
    \end{equation*}
    which is equivalent to the statement.
\end{proof}

\begin{remark}
    The preceding formula is formally equivalent to a formula found by Gaiffi in \cite[p.~330]{gaiffi2}. One can also prove \Cref{thm:HB_from_HA} by leveraging the machinery of Yuzvinsky in \cite{yuzvinsky}. We summarize in this remark the core of an alternative proof along those lines.
    Consider $\alpha=\alpha(t,x)$ and $\beta=\beta(t,x)$ as defined in \cite[Theorem~5.1]{yuzvinsky}.
    The following is equation (5.1) in Yuzvinsky's paper:
    \begin{equation}
    \alpha^{2x} -x^2 \alpha^2 = 1-x^2 -(x^2-x)t.
    \end{equation}
    This, via Manin's equation~\eqref{eq:Manin A}, is just saying that $\alpha^2 = A$ (something that Yuzvinsky notes in \cite[Remark~5.3]{yuzvinsky}).
    By differentiating this equation with respect to $t$, one obtains
    \[ \left(2 x \cdot \alpha^{2x-1} - 2x^2\alpha \right) \frac{\partial\alpha}{\partial t} = x-x^2. \]
    \[ \frac{\partial \alpha}{\partial t} = \frac{1-x}{2(\alpha^{2x-1}-x\alpha)}\]
    and \cite[eq.~(5.2)]{yuzvinsky} becomes 
    \[
    \frac{\partial \beta}{\partial t} = \frac{x}{x-1}(\alpha^{x-1}-1)\frac{\partial \alpha}{\partial t}.
    \]
    Integrating the above equation and adding the correct constant term, we obtain:
    \[
    \beta = \frac{\alpha^x-x\alpha}{x-1}+1
    \]
    and by \cite[eq.\ (5.3)]{yuzvinsky} we have 
    \[\HB = \frac{1-x}{\alpha^{x-1}-x} 
    ,\]
    which, after using again the condition $\alpha^2 = A$ yields the conclusion of the preceding theorem.  A similar proof along these lines has been also noted by Henderson in \cite[Corollary~2.5]{henderson}.
\end{remark}

By leveraging some of Manin's identities for the generating function $A$, the preceding result allows us to derive further equations relating $A$ and $B$. The following result provides two such alternatives. 

\begin{theorem}\label{thm:HB_from_HA2}
The generating functions $\HA$ and $\HB$ are related by:
\begin{align}
\HB 
&= \frac{\HA^{\frac{x+1}{2}}+x\HA}{1+x+xt}
\label{eq:second_HB_from_HA} \\
&= \frac{x \HA + \sqrt{x^2A^2 + \HA (1-x)(1+x+xt)}}{1+x+xt}.
\label{eq:third_HB_from_HA}
\end{align}
\end{theorem}

\begin{proof} 
    We start from \eqref{eq:HB_from_HA} and write
    \begin{align*}
        \HB &= \frac{1-x}{\HA^{\frac{x-1}{2}}-x} = \frac{1-x}{\HA^{\frac{x-1}{2}}-x} \cdot \frac{\HA^{\frac{x+1}{2}}+\HA x}{\HA^{\frac{x+1}{2}}+\HA x} = \frac{(1-x)(\HA^{\frac{x+1}{2}}+\HA x)}{\HA^{x}-\HA x^2}\\ 
        &= \frac{(1-x)(\HA^{\frac{x+1}{2}}+\HA x)}{1-x^2+(1-x)xt}
    \end{align*}
    where we used 
    \eqref{eq:Manin A} to replace $\HA^x-x^2\HA$ for the polynomial $1-x^2+(1-x)xt$. After simplifying the factor $1-x$ in the numerator and the denominator, we are led to \eqref{eq:second_HB_from_HA}. 
    
    To prove \eqref{eq:third_HB_from_HA}, we start from squaring \eqref{eq:second_HB_from_HA}: 
    \begin{align*}
        (1+x+xt)^2\,B^2 &= A^{x+1} + 2xA \cdot A^{\frac{x+1}{2}} + x^2A^2\\
        &= A\cdot(1-x^2+(1-x)xt+x^2A) + 2xA\cdot A^{\frac{x+1}{2}} + x^2A^2\\
        &= A(1-x^2+(1-x)xt) + 2xA(A^{\frac{x+1}{2}} + xA)\\
        &= A(1-x^2+(1-x)xt) + 2xAB(1+x+xt),
    \end{align*}
    where in the last step we employed again \eqref{eq:second_HB_from_HA}. After simplifying the common factor $1+x+xt$, we are led to the quadratic equation in $B$,
    \begin{equation}
    \label{eq:quadratic_relation_HB_HA}
    (1+x+xt)B^2 -2xAB + (x-1) A =0,
    \end{equation}
    whose solution is given precisely by \eqref{eq:third_HB_from_HA}.
\end{proof}

\begin{remark}
The formula in \eqref{eq:second_HB_from_HA} is especially useful to recursively compute the total dimension of the cohomology of $\Mbar_{0,n}^B$ in terms of the total dimension of the cohomology of $\Mbar_{0,n}$. Indeed, by setting $x=1$, one obtains:
\[ \HB(t,1)= \frac{\HA(t,1)}{1+\frac{t}{2}}. \]
This can be recast as the following alternative formulation:
\begin{align*} 
\dim \H^{\bullet}(\Mbar_{0,n+1}) &= n![t^n] A(t,1) \\
&= n![t^n](B(t,1)(1+\tfrac{t}{2})) \\
&= n!\left(\frac{1}{2^n n!} \dim \H^{\bullet}(\Mbar_{0,n}^B) + \frac{1}{2^n (n-1)!}\dim \H^{\bullet}(\Mbar_{0,n-1}^B)\right) \\
&= \frac{1}{2^n}\left(\dim\H^{\bullet}(\Mbar_{0,n}^B) + n\dim \H^{\bullet}(\Mbar_{0,n-1}^B)\right),
\end{align*}
which leads to the recursion
\[ \dim \H^{\bullet}(\Mbar_{0,n}^B) = 2^n \dim \H^{\bullet}(\Mbar_{0,n+1})  - n\dim \H^{\bullet}(\Mbar_{0,n-1}^B).\]

With this recursion at hand, together with the values of $\dim \H^{\bullet}(\Mbar_{0,n})$ tabulated in the OEIS \cite[Sequence A074059]{oeis} we can compute the first few values of $\dim \H^{\bullet}(\Mbar_{0,n}^B)$:

\[ \dim \H^{\bullet}(\Mbar_{0,n}^B) = \begin{cases} 1 & n=1,\\
2 & n=2,\\
10 & n=3,\\
72 & n=4,\\
728 & n=5,\\
9264 & n=6,\\
143792 & n=7.
\end{cases}\]
\end{remark}

\subsection{A functional equation for B}

Since Manin's result gives a functional equation for $A$ and we have several formulas that relate $A$ and $B$, we can put the ingredients together to deduce a functional equation for $B$, which is the desired type B analogue of \eqref{eq:Manin A}.

\begin{theorem}\label{thm:functional-B}
    The generating function $\HB$ satisfies the functional equation
    \begin{equation}
    \label{eq:compositional_for_B}
    \left( \frac{B^2(1+x+xt)}{1-x+2xB} \right)^{\frac{x-1}{2}} = \frac{1-x}{\HB} +x .
    \end{equation}
\end{theorem}

\begin{proof}
    From \eqref{eq:quadratic_relation_HB_HA} we can write
    \[A= \frac{(1+x+xt)B^2}{1-x+2xB},\]
    which, when substituted in \eqref{eq:HB_from_HA}, leads to the desired functional equation for $B$. Explicitly,
    \[
    \left(\frac{(1+x+xt)B^2}{1-x+2xB}\right)^{\frac{x-1}{2}}
    =A^{\frac{x-1}{2}}=x+\frac{1-x}{B}.
    \qedhere\]
\end{proof}

\subsection{A differential equation for B}

Now, we provide a differential equation for the function $B$, a counterpart to the differential equation for $A$ proved by Manin \eqref{eq:Manin diff A}.

\begin{theorem}\label{thm:diff-eq-gf-B}
The exponential generating function $B$ satisfies the differential equation
\begin{equation}
\frac{\partial \HB}{\partial t} = \frac{1}{2} \frac{B-xB+xB^2}{1+x+xt-xA}
\end{equation}
with initial condition $B(0,x)=1$.
\end{theorem}

\begin{proof}
By \eqref{eq:HB_from_HA} we write
\[
(A^{\frac{x-1}{2}}-x)B = 1 - x.
\]
After taking the partial derivative with respect to $t$ and multiplying by $AB$ on both sides we obtain
\begin{align*}
    0 &= AB\frac{\partial}{\partial t}\left((A^{\frac{x-1}{2}}-x)B\right) \\
    &= \frac{x-1}{2}A^{\frac{x-1}{2}}B^2 \frac{\partial}{\partial t}A + A(A^{\frac{x-1}{2}} - x)B\frac{\partial}{\partial t}B \\
    &= \frac{x-1}{2}(1-x+xB)B\frac{\partial}{\partial t}A +A(1-x)\frac{\partial}{\partial t}B,
\end{align*}
where in the last equality we used \eqref{eq:HB_from_HA} twice.
Dividing by $x-1$ and using \eqref{eq:Manin diff A} allows us to conclude.
\end{proof}

\subsection{A quadratic recursion}

One of the most well-known formulas to compute $A_n$ recursively stems from two (essentially equivalent) quadratic recursions which we now recall.

\begin{theorem}\label{thm:quad-A}
    The following recursive relations hold:
\begin{align*}
    A_n &= A_{n-1} + x\,\sum_{j=2}^{n-1}\binom{n-1}{j} A_j A_{n-j},\\
    &= (1+x)\,A_{n-1} + \frac{x}{2}\,\sum_{j=2}^{n-2}\binom{n}{j} A_j A_{n-j}.
\end{align*}
where $A_0 = A_1 = 1$.
\end{theorem}

The first of the above two formulas is due to Keel \cite[p.~550]{keel} whereas the second is due to Manin \cite[Corollary~0.3.2]{manin}. As we explained in \cite[Proposition~4.6]{eur-ferroni-matherne-pagaria-vecchi} it is straightforward to deduce each one from the other.

The differential equation proved in Theorem~\ref{thm:diff-eq-gf-B} lets us write the following quadratic recursion for $B_n$ that provides a type B analogue of Theorem~\ref{thm:quad-A}. The only caveat is that we need to include type A terms in the recursion.

\begin{theorem}\label{thm:quadratic-recursion}
    The following recursive relation holds:
    \begin{equation}
    B_{n+1}=B_n + x \sum_{j=1}^{n} \binom{n}{j} B_{j} B_{n-j} + x \sum_{j=2}^{n} \binom{n}{j}2^j A_{j} B_{n+1-j},
    \end{equation}
    with initial condition $B_0=1$.
\end{theorem}

\begin{proof}
   First, we can write $2 \frac{\partial}{\partial t} B = \sum_{n \geq 0} B_{n+1}\frac{t^n}{2^n\ n!}$. 
   From Theorem \ref{thm:diff-eq-gf-B} we write
   \[
   (1+x+xt)2\frac{\partial}{\partial t} B = (1-x)B + xB^2 + xA\ 2\frac{\partial}{\partial t} B\]
   and expanding every generating function and every product, we get
   \begin{align*}
   &(1+x+xt)\sum_{n \geq 0} B_{n+1}\frac{t^n}{2^n\ n!} \\
   &= \sum_{n\geq 0}\left[(1-x)B_n + x\sum_{j=0}^n\binom{n}{j}B_jB_{n-j} + x \sum_{j=0}^n\binom{n}{j}2^jA_jB_{n+1-j} \right] \frac{t^n}{2^n\ n!}.
   \end{align*}
   We differentiate both sides $n$ times with respect to $t$, set $t=0$, and multiply by $2^n$. This results in
   \[
   (1+x)B_{n+1} + 2nx B_n = (1-x)B_n + x\sum_{j=0}^n\binom{n}{j}B_jB_{n-j} + x \sum_{j=0}^n \binom{n}{j}2^jA_jB_{n+1-j},
   \]
   which can be easily simplified to the identity in the statement, after recalling that ${B_0 =1}$ and ${A_0 = A_1 = 1}$.
\end{proof}

The above formula is useful to deduce non-trivial inequalities among the Betti numbers of the cohomology of $\H^{\bullet}(\Mbar_{0,n}^B)$. Notably, Poincar\'e duality ensures that the polynomials $B_n$ are palindromic, whereas the Hard Lefschetz theorem guarantees that they are furthermore unimodal, i.e., their coefficients grow up to the middle term and then decrease.

Let $p(x)$ be a palindromic polynomial with center of symmetry $d/2$, i.e., satisfying $p(x) = x^d\,p(x^{-1})$. Then we can write (see, for example, \cite[Section~3]{branden})
\[
p(x) = \sum_{i=0}^{\lfloor d/2 \rfloor} \gamma_i x^i (1+x)^{d-2i}.
\]
We call the polynomial $\gamma(p,x) \coloneq \sum_{i=0}^{\lfloor d/2 \rfloor}\gamma_i x^i$ the \emph{$\gamma$-polynomial of $p$} and we say that $p$ is \emph{$\gamma$-positive} if the coefficients of $\gamma(p,x)$ are non-negative. Being $\gamma$-positive is a desirable property to satisfy, as it sits in the middle of the following chain of implications. For a symmetric polynomial $p$,
\[
\text{real-rootedness} \implies \text{$\gamma$-positivity} \implies \text{unimodality}.
\]

The type A version of the following result is proved by Aluffi, Chen, and Marcolli in \cite[Theorem~1.2]{aluffi-marcolli-chen}.

\begin{corollary}
The polynomials $B_n(x)$ are $\gamma$-positive for every $n$.
\end{corollary}
\begin{proof}
We prove this by strong induction on $n$, with the base case being trivial. By Theorem \ref{thm:quadratic-recursion}, we can write
\[
B_{n+1} = (1+x)B_{n} + x\sum_{j=1}^{n-1}\binom{n}{j}B_jB_{n-j} + x\sum_{j=2}^{n}\binom{n}{j}2^jA_jB_{n+1-j}.
\]
Since $\deg(B_j) = j-1$ for $j\geq 1$ and $\deg(A_j) = j-2$ for $j\geq 2$, each individual term on the right-hand side is palindromic with center of symmetry $n/2$. Indeed,
\begin{align*}
\operatorname{center}((1+x)B_n)&=\tfrac12+\tfrac{n-1}{2}=\tfrac n2,\\
\operatorname{center}(xB_jB_{n-j})&=1+\tfrac{j-1}{2}+\tfrac{n-j-1}{2}=\tfrac n2,\\
\operatorname{center}(xA_jB_{n+1-j})&=1+\tfrac{j-2}{2}+\tfrac{n-j}{2}=\tfrac n2.
\end{align*}
Therefore, we can write the following identity for the $\gamma$-polynomials.
\[
\gamma_{n+1}^B = \gamma_{n}^B + x\sum_{j=1}^{n-1}\binom{n}{j}\gamma_j^B\gamma_{n-j}^B + x\sum_{j=2}^{n}\binom{n}{j}2^j\gamma_j^A\gamma_{n+1-j}^B.
\]
In the above equation we use the shorthand $\gamma^A_n(x) = \gamma(A_n,x)$ and $\gamma_n^B(x):= \gamma(B_n,x)$ for each $n\geq 1$.
By strong induction and by \cite[Theorem~1.2]{aluffi-marcolli-chen}, every polynomial on the right-hand side has positive coefficients, therefore $\gamma_{n+1}^B$ has positive coefficients as well. 
\end{proof}

The following are the first few values of $\gamma_n^B(x)$.

\[ \gamma_n^B(x) = \begin{cases}
1 & n= 1 ,\\
1 & n= 2 ,\\
6 \, x + 1 & n= 3 ,\\
32 \, x + 1 & n= 4 ,\\
240 \, x^{2} + 118 \, x + 1 & n= 5 ,\\
3096 \, x^{2} + 380 \, x + 1 & n= 6 ,\\
25200 \, x^{3} + 25000 \, x^{2} + 1158 \, x + 1 & n= 7 ,\\
599808 \, x^{3} + 164320 \, x^{2} + 3456 \, x + 1 & n= 8 ,\\
5080320 \, x^{4} + 8491968 \, x^{3} + 968224 \, x^{2} + 10254 \, x + 1 & n= 9.\\
\end{cases}\]

A very intriguing conjecture by Aluffi, Chen, and Marcolli appearing in \cite[Conjecture~1.1]{aluffi-marcolli-chen} postulates that the polynomials $A_n$ are real-rooted for every $n\geq 1$. For a palindromic polynomial with nonnegative coefficients, it is known that real-rootedness implies both $\gamma$-positivity and (ultra) log-concavity. The last corollary makes it natural to propose the following type B analogue of the conjecture\footnote{After this manuscript was finished and posted, the conjecture in \cite{aluffi-marcolli-chen} was solved by B\'erczi and Kiem \cite{berczi-kiem}.} in \cite{aluffi-marcolli-chen}.

\begin{conjecture}
    For all $n\geq 1$ the polynomials $B_n$ are real-rooted.
\end{conjecture}

With the help of a computer, we have verified that $B_n$ is real-rooted for every $n\leq 300$.

\subsection{Compositional formulas for B}

Yet another classical formula for the generating function $A$ was found by Getzler in \cite[p.~228]{Getzler94}. We recapitulate Getzler's formula now.

\begin{theorem}
    The generating function $A(t,x)$ satisfies the compositional identity:
    \[A\left(t-\frac{(1+t)^x - 1-xt}{x(x-1)},x\right) = 1+t.\]
\end{theorem}

By employing our functional equation, we may derive type B versions of Getzler's result, each depending on which part of the formula we want to preserve.

\begin{theorem}
The generating function $B(t,x)$ satisfies the compositional identity:
\[ \HB\left(t-\frac{(1+t)^x-1-xt}{x(x-1)},x\right) =  \frac{x-1}{x-(1+t)^{\frac{x-1}{2}}}.\]
Furthermore, it also satisfies the following compositional identity:
    \[B \left( \frac{1}{x}\left( \frac{1-x+2xt}{t^2} \left( \frac{1-x}{t}+x \right) ^{\frac{2}{x-1}}-1-x\right),x \right)=t.\]
\end{theorem}
\begin{proof}
    To prove the first one, let us write 
    \[ Y(t,x) = t-\frac{(1+t)^x-1-xt}{x(x-1)}. \] 
    By Lemma~\ref{lemma:gen_fun_char_poly_B} the statement is then equivalent to $\HB(Y,x)=-\frac{1}{\overline{\chi}^B(t,x)}$. This can be proven by rewriting~\eqref{eq:HB_from_HA} as ${\HB(t,x)=-\frac{1}{\overline{\chi}^B(\HA-1,x)}}$ and by observing that $A(Y,x)=1+t$, see \cite[p.~228]{Getzler94} or \cite[Proposition~4.5]{eur-ferroni-matherne-pagaria-vecchi}, so that
    \[
    B(Y,x) = -\frac{1}{\overline{\chi}^B(A(Y,x)-1,x)} = -\frac{1}{\overline{\chi}^B(t,x)},
    \]
    as required.
    Let us prove the second identity now. From \eqref{eq:compositional_for_B} we express $t$ in terms of $B$, obtaining
    \[ t= \frac{1}{x}\left( \frac{1-x+2xB}{B^2} \left( \frac{1-x}{B}+x \right) ^{\frac{2}{x-1}}-1-x\right).\]
    Since left and right compositional inverses of a formal power series must coincide, we obtain the second compositional formula.
\end{proof}

\section{Iterative formulas for type B}

We close this paper by discussing an iterative (i.e., non-recursive) formula for computing the polynomials $B_n(x)$. In \cite{eur-ferroni-matherne-pagaria-vecchi} we provide several iterative formulas for the Chow polynomials of matroids with arbitrary building sets. In the special case of the partition lattice of type B and the minimal building set, the result can actually be expressed in a very satisfying way by means of the Bell polynomials.

\begin{theorem}\label{thm:iterative-B}
For every $n$, the following formula holds:
    \[
B_n(x) = \sum_{k=1}^n 2^{n-k}\left( \sum_{\ell=0}^k \ell! \Bell_{k,\ell} (\overline{\chi}_{1}^B, \overline{\chi}_{2}^B, \dots ) \right)  \Bell_{n,k} (A_1 ,A_2, \dots)
    \]
\end{theorem}
\begin{proof}
    By~\eqref{eq:HB_from_HA} we write again $B(t,x) = -\frac{1}{\overline{\chi}^B(A-1,x)}$. The signs cancel as follows:
    \[
    U(t)=\sum_{m\geq1}\overline{\chi}_m^B(x)\frac{t^m}{2^m m!},\qquad
    -\frac{1}{\overline{\chi}^B(t,x)}
    =\frac{1}{1-U(t)}=\sum_{\ell\geq0}U(t)^\ell.
    \]
    \[
    2^k k!\,[t^k]\left(-\frac{1}{\overline{\chi}^B(t,x)}\right)
    =\sum_{\ell=0}^k\ell!\,\Bell_{k,\ell}
    (\overline{\chi}_1^B,\overline{\chi}_2^B,\ldots).
    \]
    The theorem follows by using first the geometric series to obtain a formula for $\frac{1}{\overline{\chi}^B}$ and then Corollary~\ref{coro:compositional_formula_type_B}.
\end{proof}

In type $A$, since $A-1$ is the compositional inverse of $-\overline{\chi}^A(t,x)$, we can apply Theorem~\ref{thm:compositional_formula} to obtain
\[ A_n = \sum_{k=0}^{n-1} \Bell_{n+k-1,k}(0,\overline{\chi}^A_2,\ldots, \overline{\chi}^A_{n-k+1}).\]
This formula, combined with Poincar\'e duality (i.e., the fact that $A_n$ is a palindromic polynomial) and an elementary identity involving Stirling numbers of the first and the second kind \cite[Equation~(8)]{eur-ferroni-matherne-pagaria-vecchi} allows one to deduce the following closed formula for $A_n$.

\begin{theorem}\label{thm:amn}
    For every $n$ the following formula holds:
    \[ A_n(x) = (1-x)^n \sum_{k\geq 0}\sum_{j\geq 0} s(k+n,k+n-j)S(k+n-j,k+1) x^{k+j},\]
    where $s(n,k)$ and $S(n,k)$ denote, respectively, the signed Stirling numbers of the first kind, and the Stirling numbers of the second kind.
\end{theorem}

The above result was established first in work of Aluffi, Marcolli, and Nascimento \cite{aluffi-marcolli-nascimento}, and later reproved in \cite[Theorem~1.5]{eur-ferroni-matherne-pagaria-vecchi}. In particular, the formula appearing in Theorem~\ref{thm:iterative-B} might be regarded as a type B analogue of the formula by Aluffi, Marcolli, and Nascimento. 

Nonetheless, it would be desirable to express $B_n$ in terms of a reasonably simple combination of the Whitney numbers of the first and second kind of the partition lattice $\Pi_n^B$ (note that in the type A case the Whitney numbers of $\Pi_n$ of both kinds are exactly the classical Stirling numbers). Despite several attempts, we have not been able to establish such a simple formula, which we propose as a challenge to interested readers.

\begin{problem}
    Find a closed expression for $B_n$, akin to Theorem~\ref{thm:amn}, in terms of the Whitney numbers of both kinds of the type B partition lattices.
\end{problem}

\subsection*{Acknowledgments}

We thank our co-authors, Chris Eur and Jacob P. Matherne, who collaborated with us in the prequel paper \cite{eur-ferroni-matherne-pagaria-vecchi}. 
Luis Ferroni and Roberto Pagaria are members of the INDAM research group GNSAGA -- Gruppo Nazionale per le Strutture Algebriche, Geometriche e le loro Applicazioni.
Roberto Pagaria is partially supported by the Italian Ministry of University and Research (MUR) under the grant FIS 2, project “Chow Ring of Toric Arrangements” (FIS-2023-00808, CUP J53C25001900001).

\bibliographystyle{amsalpha}
\bibliography{bibliography}

\end{document}